\documentclass[reqno,a4paper,11pt]{article}

\usepackage{a4wide}
\usepackage[centertags]{amsmath}
\usepackage{eucal,dsfont,accents,bbm,amsfonts,amssymb,amsthm,amsopn}

\usepackage[usenames]{color}
\definecolor{citegreen}{rgb}{0,0.6,0}
\definecolor{refred}{rgb}{0.8,0,0}
\usepackage[colorlinks, citecolor=citegreen, linkcolor=refred]{hyperref}

\title{Gaussian upper bounds for the heat kernel on\\ evolving manifolds}
\author{Reto Buzano and Louis Yudowitz}
\date{}

\providecommand{\abs}[1]{\lvert #1\rvert}

\providecommand{\norm}[1]{\lVert #1\rVert}

\DeclareMathOperator{\Vol}{Vol}

\newcommand{\Rm}{\mathrm{Rm}}
\newcommand{\Rc}{\mathrm{Rc}}
\newcommand{\Sc}{\mathrm{Sc}}

\newcommand{\RR}{\mathbb{R}}

\newcommand{\eps}{\varepsilon}
\newcommand{\Lap}{\triangle}
\newcommand{\dt}{\tfrac{\partial}{\partial t}}

\newtheoremstyle{break}%
  {12pt}%
  {16pt}%
  {\itshape}%
  {}%
  {\bfseries}%
  {}%
  {\newline}%
  {\thmname{#1}\thmnumber{ #2}\thmnote{ \normalfont{(#3)}}}%

\theoremstyle{break}%
\newtheorem{lemma}{Lemma}[section]%
\newtheorem{mainthm}[lemma]{Main Theorem}%
\newtheorem{thm}[lemma]{Theorem}%
\newtheorem{cor}[lemma]{Corollary}%
\newtheorem{prop}[lemma]{Proposition}%
\newtheorem*{rem}{Remark.}%
\numberwithin{equation}{section}%

\newcommand\printaddress{{
\setlength{\parindent}{15pt}
\footnotesize~
\par
{\scshape Reto Buzano}
\newline 
Queen Mary University of London, 
School of Mathematical Sciences, 
Mile End Road, 
London E1 4NS, UK 
\newline
\textit{E-mail address:} 
\texttt{r.buzano@qmul.ac.uk}
\newline 
Universit\`a degli Studi di Torino, 
Dipartimento di Matematica,
Via Carlo Alberto 10, 
10123 Torino, Italy 
\newline
\textit{E-mail address:} 
\texttt{reto.buzano@unito.it}
\par
{\scshape Louis Yudowitz}
\newline 
Queen Mary University of London, 
School of Mathematical Sciences, 
Mile End Road, 
London E1 4NS, UK
\newline
\textit{E-mail address:} 
\texttt{l.yudowitz@qmul.ac.uk}
\par
}}

\begin{document}
\maketitle

\begin{abstract}
In this article, we prove a general and rather flexible upper bound for the heat kernel of a weighted heat operator on a closed manifold evolving by an intrinsic geometric flow. The proof is based on logarithmic Sobolev inequalities and ultracontractivity estimates for the weighted operator along the flow, a method which was previously used by Davies \cite{D87} in the case of a non-evolving manifold. This result directly implies Gaussian-type upper bounds for the heat kernel under certain bounds on the evolving distance function; in particular we find new proofs of Gaussian heat kernel bounds on manifolds evolving by Ricci flow with bounded curvature or positive Ricci curvature. We also obtain similar heat kernel bounds for a class of other geometric flows.
\end{abstract}

\section{Introduction and Main Results}\label{sec.intro}
This article is concerned with heat kernel estimates on evolving manifolds, but we start with a brief discussion of such bounds for static manifolds. To this end, let $(M^n,g)$ be a complete Riemannian manifold of dimension $n \geq 3$ and consider the \emph{heat kernel} or \emph{fundamental solution} $H(x,t;y,s)$, i.e. the minimal solution of
\begin{equation}\label{eq.kernel}
\begin{split}
(\dt-\Lap_x)H(x,t;y,s) &=0,\\
\lim_{t\searrow s} H(\cdot,t;y,s) &=\delta_y,
\end{split}
\end{equation}
for $x,y\in M$ and $t>s$. Here, $\Lap_x$ denotes the Laplace-Beltrami operator with respect to the metric $g$ in the $x$-variable and the limit to the Dirac-$\delta$ based at $y$ has to be understood in the sense of measures. It is well known that on Euclidean $\RR^n$, the heat kernel is given by the explicit formula
\begin{equation*}
H(x,t;y,s)=\frac{1}{[4\pi(t-s)]^{n/2}}\,e^{-\frac{\abs{x-y}^2}{4(t-s)}}.
\end{equation*}
On Riemannian manifolds, bounds of similar type were first obtained by Cheng-Li-Yau \cite{CLY81} in the case of complete manifolds with bounded sectional curvature and further improved by Li-Yau \cite{LY86} using their famous differential Harnack inequalities. Under a certain curvature assumption, they proved that 
\begin{equation}\label{eq.offdiag}
H(x,t;y,s)\leq\frac{C}{f(t-s)}\,e^{-\frac{d^2(x,y)}{D(t-s)}},
\end{equation}
where $C$ and $D$ are sufficiently large constants and $f(\cdot)$ is an increasing function. (In fact, they showed that $D$ can be chosen arbitrarily close to the optimal value $4$.) A bound of the form \eqref{eq.offdiag} is usually referred to as a \emph{Gaussian} upper bound or \emph{off-diagonal bound} and it directly implies the (logically weaker) \emph{on-diagonal bound}
\begin{equation}\label{eq.ondiag}
H(x,t;y,s)\leq\frac{C}{f(t-s)}.
\end{equation}
Surprisingly, in many situations the bounds in \eqref{eq.offdiag} and \eqref{eq.ondiag} turn out to be equivalent! A beautiful, abstract theory exploring this fact was developed by Davies in a series of papers \cite{DS84,D87,D89,DP89} where he provided a method to obtain Gaussian upper bounds from on-diagonal bounds on quite general manifolds, using logarithmic Sobolev inequalities introduced by Gross \cite{G75}. Compared to previous work, his method has the advantage that it does not directly depend on any curvature assumptions for the underlying Riemannian manifold. Let us mention that around the same time similar methods using different functional inequalities were developed. To summarise, these prove in particular that an on-diagonal bound \eqref{eq.ondiag} with $f(t)=t^{n/2}$, where $n=\dim M$, is equivalent to any of the following functional inequalities, each of them also implying an off-diagonal upper bound \eqref{eq.offdiag}: 
\begin{itemize}
\item a logarithmic Sobolev inequality (see Davies \cite{D87}),
\item a proper Sobolev inequality (see Varopoulos \cite{V85}),
\item a Nash type inequality (see Carlen-Kusuoka-Stroock \cite{CKS87}),
\item or a Faber-Krahn type inequality (see Carron \cite{C94} and Grigor'yan \cite{G94}).
\end{itemize}

Finally, Grigor'yan \cite{G97} developed a direct method to deduce off-diagonal upper bounds from on-diagonal ones without using a bridging functional inequality and allowing a large class of functions $f(t)$. In particular, his result extends work of Ushakov \cite{U80} who first proved that \eqref{eq.ondiag} implies \eqref{eq.offdiag} on Euclidean space and for polynomial $f(t)$.\\

Let us now discuss the case where the underlying Riemannian manifold is not fixed (and thus the Laplace operator used in the definition of the heat kernel in \eqref{eq.kernel} is time-dependent). In 2002, Guenther \cite{G02} proved existence of a fundamental solution on a compact manifold with a smoothly time-dependent metric $g(t)$. Since then, and in particular motivated by the work of Perelman \cite{P02} who developed important Harnack inequalities and monotone quantities for solutions of the (adjoint) heat equation on a manifold evolving by the Ricci flow, many authors have proved Gaussian-type upper bounds for the heat kernel on such evolving manifolds.\\

In the case where $(M,g(t))$ evolves by Hamilton's Ricci flow $\dt g=-2\Rc$ and has uniformly bounded curvature in space-time, the direct method of Grigor'yan \cite{G97} can be adopted with some modifications (see Chau-Tam-Yu \cite{CTY11} or Theorem 26.25 in the Ricci flow book \cite{RF3}). The result can be stated as follows.

\begin{thm}[cf.~Chau-Tam-Yu \cite{CTY11}, Chow et al. \cite{RF3}]\label{thm.cty}
Let $(M^n,g(t))$ be a solution to the Ricci flow with $n\geq 3$ and with uniformly bounded curvature on $[0,T]$, $T<\infty$. Then there exists a constant $C$ depending on $n$, $T$, and $\sup_{M\times[0,T]}\abs{\Rm}$ such that the heat kernel satisfies
\begin{equation*}
H(x,t;y,s)\leq\frac{C}{(t-s)^{n/2}} \,e^{-\frac{d_{g(t)}^2(x,y)}{C(t-s)}},
\end{equation*}
for any $x,y\in M$ and $0\leq s<t\leq T$. 
\end{thm}

Let us remark that in this case where the curvature is uniformly bounded along the Ricci flow all the metrics $g(t)$ are uniformly equivalent and we could therefore use the distance function with respect to a fixed metric, e.g. $g(0)$, by possibly changing the constant $C$.\\

However, bounds on the (adjoint) heat kernel on a Ricci flow seem particularly interesting near points where the curvature tends to infinity, since they can then be used to understand the singular behaviour of the flow (for example by using Perelman's $\mathcal{W}$-entropy). An important step in this direction was made by Cao-Zhang \cite{CZ10}. They proved an on-diagonal bound without curvature assumptions using a uniform logarithmic Sobolev inequality along the Ricci flow, as found for example in the works of Ye \cite{Y07} and Zhang \cite{Z07,Z10} (see also B\u{a}ile\c{s}teanu \cite{B12} for a similar on-diagonal bound). In the same paper, Cao-Zhang also obtained off-diagonal bounds (using again Grigor'yan's direct method) under the assumption of \emph{positive Ricci curvature}. Their result is the following.

\begin{thm}[cf.~Cao-Zhang \cite{CZ10}]\label{thm.cz}
Let $(M^n,g(t))$ be a solution to the Ricci flow on $[0,T)$, $T<\infty$ in dimension $n\geq 3$. Assume $g(t)$ has nonnegative Ricci curvature for all times and that it is not Ricci-flat. Then there exists a constant $C$ depending on $n$, $T$, and $g(0)$, as well as a numerical constant $\eta$, such that the fundamental solution of the heat equation satisfies
\begin{equation*}
H(x,t;y,s)\leq\frac{C}{(t-s)^{n/2}} \,e^{-\eta\Lambda(t)} \,e^{-\frac{d_{g(t)}^2(x,y)}{C(t-s)}},
\end{equation*}
for any $x,y\in M$ and $0\leq s<t< T$. Here $\Lambda(t):=\int_0^t \min_M R(\cdot,\lambda)d\lambda$.
\end{thm}

This result allowed the authors to classify blow-down limits of so-called Type~I $\kappa$-solutions of the Ricci flow. A similar result for Ricci flows with Ricci curvature bounded below has been obtained by Zhu in \cite{zhu}, relying on double integral estimates. Other Gaussian bounds have been obtained for example for Type~I Ricci flows by Mantegazza and the first author \cite{MM15} or, with a much more elaborate proof, for Ricci flows with bounded scalar curvature by Bamler and Zhang \cite{BZta}. The latter two results rely on a different type of logarithmic Sobolev inequality found by Hein and Naber \cite{HN14}, a Gaussian lower bound for the heat kernel and the parabolic mean value inequality. Specific applications of these bounds include showing blow-up limits of Type~I singularities of the Ricci flow are non-trivial gradient shrinking Ricci solitons and proving weak convergence results for the Ricci flow when the scalar curvature is uniformly bounded. Finally, in a recent preprint \cite{wu}, Wu obtained a sharp Gaussian bound for the (Schr\"odinger) heat kernel on shrinking Ricci solitons.\\

The goal of the present article is to develop a general approach to proving Gaussian-type heat kernel bounds that work in a variety of different situations and only rely on the behaviour of the distance function rather than \emph{explicitly} on curvature assumptions. In contrast to the proofs of the theorems above, we neither use Grigor'yan's direct method (as in the original proofs of Theorems \ref{thm.cty} and \ref{thm.cz}) nor the Hein-Naber Sobolev inequality or mean value inequality (as in \cite{MM15} and \cite{BZta}). Instead, we use the ideas of Davies \cite{D87} of proving the Gaussian upper bounds using a bridging functional inequality and showing ultracontractivity estimates for a \emph{weighted} operator. We will see that this method can be used to find new proofs of (variants of) Theorem \ref{thm.cty} and \ref{thm.cz}.\\

The main effort of this article goes into proving the following key theorem.

\begin{mainthm}[Upper bounds for the kernel of a weighted heat operator]
\label{main.theorem}
Let $(M^n,g(t))$ be a compact solution to the Ricci flow on $[0,T)$, $T<\infty$ in dimension $n\geq 3$. Then there exists a constant $C$ depending only on $n$, $T$ and $g(0)$ such that the following holds. Let $\psi:M\times[0,T)\to\RR$ be a smooth function with $\psi_t(\cdot)=\psi(\cdot,t)$ satisfying $\abs{\nabla\psi_t}_{g(t)}\leq 1$ and let $K(x,t;y,s)$ be the fundamental solution of the weighted heat operator $\dt-L_t$, where $L_t=\phi_t^{-1}\Lap_{g(t)}\phi_t$ for $\phi_t=e^{\alpha\psi_t}$ with $\alpha\in\RR$. Then, we have the upper bound
\begin{equation}\label{eq.mainlemma}
K(x,t;y,s)\leq \frac{C}{(t-s)^{n/2}}\,e^{2\alpha^2(t-s)},
\end{equation}
for all $x,y\in M$ and $0\leq s<t<T$.
\end{mainthm}

Clearly, setting $\alpha=0$ and $\psi_t \equiv 1$, we obtain the on-diagonal bound
\begin{equation}
H(x,t;y,s)\leq\frac{C}{(t-s)^{n/2}}
\end{equation}
for the fundamental solution of the heat equation on a manifold evolving by Ricci flow without any curvature assumption. For suitable choices of $\alpha$ and $\psi_t$, we can also obtain Gaussian-type upper bounds. Two slightly different such bounds are given in the following corollary.

\begin{cor}[Gaussian upper bounds for the heat kernel along the Ricci flow]
\label{cor.main}
Let $(M^n,g(t))$ be a compact solution to the Ricci flow on $[0,T)$, $T<\infty$ in dimension $n\geq 3$. Then there exists a constant $C$ depending only on $n$, $T$, and the initial manifold $(M,g(0))$, such that the fundamental solution of the heat equation satisfies the following estimates.
\begin{itemize}
\item[i)] For any $x,y\in M$ and $0\leq s<t<T$
\begin{equation}\label{main.eq1}
H(x,t;y,s)\leq\frac{C}{(t-s)^{n/2}}\,e^{-\frac{d_{g(t)}^2(x,y)}{8\mu^2(t-s)}},
\end{equation}
where
\begin{equation*}
\mu:=\sup_{\lambda\in[s,t]}\,\sup_{M\setminus L}\,\abs{\nabla d_{g(t)}(y,\cdot)}_{g(\lambda)},
\end{equation*}
and $L$ is the set where $d_{g(t)}(y,\cdot)$ is not differentiable.
\item[ii)] Furthermore, for any $x,y\in M$ and $0\leq s<t<T$
\begin{equation}\label{main.eq2}
H(x,t;y,s)\leq\frac{C}{(t-s)^{n/2}}\,e^{-\frac{d_{g(t)}^2(x,y)}{8(t-s)}\,+\,\eta d_{g(t)}(x,y)},
\end{equation}
where
\begin{equation*}
\eta:=\tfrac{1}{4}\sup_{\lambda\in[s,t]}\sup_{\phantom{[}z\phantom{]}}\;\max\big\{\tfrac{\partial}{\partial\sigma} d_{g(\sigma)}(z,y)\big|_{\sigma=\lambda}, 0\big\},
\end{equation*}
where the second supremum is taken over all $z$ with $d_{g(\lambda)}(z,y)\leq d_{g(\lambda)}(x,y)$.
\end{itemize}
\end{cor}

We note that the bounds in this corollary depend on the behaviour of the distance function along the flow rather than directly involving curvature bounds. It is however now very easy to give new proofs of (variants of) the Theorems \ref{thm.cty} and \ref{thm.cz}. In fact, in the case where the sectional curvature is uniformly bounded in space-time, we obtain an uniform bound for $\mu$ in \eqref{main.eq1}, and thus a result as in Theorem \ref{thm.cty}. In the case where $\Rc\geq 0$, we have $\dt d_{g(t)}(x,y)\leq 0$ along the Ricci flow and hence $\eta=0$ in \eqref{main.eq2}, that is, we obtain a Gaussian upper bound similar to Theorem \ref{thm.cz}.\\

We would like to point out that Gaussian-type lower bounds have been previously obtained without curvature assumptions by Cao-Zhang \cite{CZ10} based on Harnack inequalities proved by Zhang \cite{Z06} and Cao-Hamilton \cite{CH09}. Moreover, in many situations they also follow from an estimate of Perelman's reduced length functional (see e.g. \cite{MM15} for Type~I flows or \cite{BZta, Z12} for flows with bounded scalar curvature). Therefore, we restrict ourselves to proving upper bounds here.\\

In the second part of the paper, we discuss other geometric flows of the form $\dt g=-2\Sc$, where $\Sc=(S_{ij})$ is a symmetric two-tensor with trace $S = g^{ij}S_{ij}$. We will always assume that for each vector field $X$ on $M$ we have the following tensor inequality
\begin{equation}\label{eq.D}
\begin{aligned}
0\leq \mathcal{D}(\Sc,X) &:= \dt S-\Lap S -2\abs{S_{ij}}^2 +4(\nabla_i S_{ij})X_j -2(\nabla_j S)X_j \\
&\quad + 2R_{ij}X_iX_j - 2S_{ij}X_iX_j.
\end{aligned}
\end{equation}
The main result for such flows is the following variant of Theorem \ref{main.theorem} and Corollay \ref{cor.main}.

\begin{thm}[Gaussian bounds for the heat kernel along geometric flows with $\mathcal{D}(\Sc,X)\geq 0$]
\label{thm.mainS}
Let $n\geq 3$ and let $(M^n,g(t))$ be a compact solution to $\dt g=-2\Sc$ on $[0,T)$, $T<\infty$ satisfying \eqref{eq.D}. Then all the bounds from Theorem \ref{main.theorem} and Corollary \ref{cor.main} still hold. 
\end{thm}

Geometric flows satisfying the inequality $\mathcal{D}(\Sc,X) \geq 0$, $\forall X\in\Gamma(TM)$, were first studied by the first author in \cite{M10}. Apart from the Ricci flow where $\mathcal{D}(\Rc,X)\equiv 0$, this inequality is for example satisfied by non-evolving manifolds of nonnegative Ricci curvature, List's extended Ricci flow system \cite{L08}, the harmonic Ricci flow \cite{M12}, the twisted K\"ahler-Ricci flow \cite{CS16} on Fano manifolds, or the Lorentzian mean curvature flow \cite{H99} on Lorentzian manifolds of nonnegative sectional curvatures. In particular, Theorem \ref{thm.mainS} gives Gaussian bounds for all of these flows.\\

The article is organised as follows. In Section \ref{sec.Ricci}, we prove Theorem \ref{main.theorem} and Corollary \ref{cor.main} using logarithmic Sobolev inequalities along the Ricci flow and ultracontractivity estimates for a weighted heat operator. In Section \ref{sec.OtherFlows} we explain the proof of Theorem \ref{thm.mainS}.\\

\textbf{Acknowledgements.} We would like to thank Gianmichele Di Matteo and Shengwen Wang for interesting discussions. The first author has been supported by the EPSRC grants EP/M011224/1 and EP/S012907/1. The second author has been supported by a studentship from the QMUL Faculty of Science and Engineering Research Support Fund.

\section{Gaussian bounds along the Ricci flow}\label{sec.Ricci}

In this section, we prove heat kernel bounds along the Ricci flow following the strategy of Davies for non-evolving manifolds \cite{D87}. As a first step towards Theorem \ref{main.theorem}, we prove $L^p$-logarithmic Sobolev inequalities for a weighted Laplacian along the Ricci flow in Subsection \ref{subsec.Sob}. These inequalities are then used in Subsection \ref{subsec.ultra} to obtain ultracontractivity estimates for a weighted heat operator allowing to estimate the $L^\infty$ norm of a solution $u$ at some time $t_1$ by the $L^2$ norm at an earlier time $t_0$, see Lemma \ref{lemma.ultra}. Finally, in Subsection \ref{subsec.L12} we prove a similar contraction estimate from $L^1$ to $L^2$, see Lemma \ref{lemma.L12}. This step follows from the second step by a simple duality argument in the work of Davies, but needs a new argument when the underlying manifold is evolving. In the last subsection, we combine the contraction estimates to give a proof of Theorem \ref{main.theorem} and Corollary \ref{cor.main}.

\subsection{Log-Sobolev Inequalities for Weighted Laplacian}\label{subsec.Sob}
Let us recall the uniform logarithmic Sobolev inequality along the Ricci flow proved by Ye \cite{Y07} and Zhang \cite{Z07,Z10}. See their articles for precise definitions of $A$ and $B$ in the Proposition below.

\begin{prop}[Uniform Log-Sobolev inequality along the Ricci flow, cf. \cite{Y07,Z07,Z10}]\label{prop.logsob}
Let $(M^n,g(t))$ be a compact solution to the Ricci flow $\dt g(t)=-2\Rc_{g(t)}$ in dimension $n\geq 3$ on some positive time interval $[0,T)$, $T<\infty$. For all $\eps>0$ and each $t\in[0,T)$, there holds
\begin{equation*}
\int_M v^2\log v^2\, dV_{g(t)} \leq \eps\int_M\big(\abs{\nabla v}^2+\tfrac{1}{4} R_{g(t)}v^2\big)dV_{g(t)} +\gamma(\eps,t),
\end{equation*}
for all $0\leq v\in C^{\infty}_c(M)$ with $\norm{v}_2=1$. Here, $R_{g(t)}$ denotes the scalar curvature of $(M,g(t))$ and 
\begin{equation*}
\gamma(\eps,t):=-\tfrac{n}{2}\log \eps + A+B\big(t+\tfrac{\eps}{4}\big),
\end{equation*}
where $A$, $B$ are constants depending only on $(M,g(0))$.
\end{prop}

We point out that we will prove a more general version of this result in Section \ref{sec.OtherFlows}, see Proposition \ref{prop.sobS}. The $L^2$-norm $\norm{v}_2$ in this proposition and all the $L^p$-norms $\norm{\cdot}_p$ in the following are taken with respect to the (time-dependent) volume element $dV_{g(t)}$. An easy consequence of the Ye-Zhang logarithmic Sobolev inequality is the following $L^p$-logarithmic Sobolev inequality.

\begin{lemma}[$L^p$-logarithmic Sobolev inequality along the Ricci flow]\label{lemma.logsob}
Let $(M^n,g(t))$ be a solution to the Ricci flow in dimension $n\geq 3$ on $[0,T)$, $T<\infty$. For $\eps>0$, $t\in[0,T)$, $1< p<\infty$, and $0\leq u\in C^{\infty}_c(M)$, there holds
\begin{align*}
\int_M u^p\log u\, dV_{g(t)} &\leq -\tfrac{\eps}{2}\int_M u^{p-1}\Lap u\,dV_{g(t)} +\tfrac{p-1}{2p^2}\eps\int_M R_{g(t)}u^p\,dV_{g(t)}\\
&\quad\,+\widetilde{\gamma}(\eps,p,t)\norm{u}_p^p +\norm{u}_p^p\log\norm{u}_p,
\end{align*}
where 
\begin{equation*}
\widetilde{\gamma}(\eps,p,t):=\tfrac{1}{p}\big({-\tfrac{n}{2}}\log\big(\tfrac{2(p-1)}{p}\eps\big) + A+B\big(t+\tfrac{(p-1)}{2p}\eps\big)\big)
\end{equation*}
with $A$, $B$ as in Proposition \ref{prop.logsob} above.
\end{lemma}

\begin{proof}
Define $v:= \frac{u^{p/2}}{\norm{u^{p/2}}_2}$ (such that $0\leq v\in C^{\infty}_c(M)$ with $\norm{v}_2=1$). Since $\norm{u^{p/2}}_2^2=\norm{u}_p^p$, we find
\begin{equation*}
v^2\log v^2= \tfrac{u^p}{\norm{u}_p^p}\log\Big(\tfrac{u^p}{\norm{u}_p^p}\Big)= \tfrac{p \, u^p}{\norm{u}_p^p}\big(\log u-\log\norm{u}_p\big)
\end{equation*}
and thus
\begin{equation*}
\int_M v^2\log v^2\,dV_{g(t)} =\tfrac{p}{\norm{u}_p^p}\int_M u^p\log u\,dV_{g(t)} -p\log\norm{u}_p
\end{equation*}
Proposition \ref{prop.logsob} applied to $v$ then yields
\begin{align*}
\int_M u^p\log u\,dV_{g(t)} &=\tfrac{\norm{u}_p^p}{p}\Big(\int_M v^2\log v^2\,dV_{g(t)}+p\log\norm{u}_p\Big)\\
& \leq \tfrac{\norm{u}_p^p}{p}\Big(\widetilde{\eps}\int_M\big(\abs{\nabla v}^2+\tfrac{1}{4} R_{g(t)}v^2\big)dV_{g(t)} +\gamma(\widetilde{\eps},t)\Big)+\norm{u}_p^p\log\norm{u}_p\\
&=\tfrac{p\widetilde{\eps}}{4(p-1)}\int_M \nabla u^{p-1}\nabla u\, dV_{g(t)} +\tfrac{\widetilde{\eps}}{4p}\int_M R_{g(t)}u^p\,dV_{g(t)}\\
&\quad\,+\tfrac{\gamma(\widetilde{\eps},t)}{p}\norm{u}_p^p+\norm{u}_p^p\log\norm{u}_p.
\end{align*}
The corollary then follows by setting $\eps:=\frac{p\widetilde{\eps}}{2(p-1)}$.
\end{proof}

Following Davies \cite{D87}, we now introduce the \emph{weighted operator} $L:=\phi^{-1}\Lap_g\phi$ on a manifold $(M,g)$, where $\phi=e^{\alpha\psi}$ with $\alpha\in\RR$ and $\psi:M\to\RR$ satisfying $\abs{\nabla\psi}_g\leq 1$. We have the following estimate.

\begin{lemma}[cf. Davies \cite{D87}]\label{lemma.davies1}
For every complete manifold $(M^n,g)$, $0\leq u\in C^{\infty}_c(M)$, $2\leq p<\infty$, and $L=\phi^{-1}\Lap_g\phi$ as above, we have
\begin{equation*}
2\int_M u^{p-1}L u\,dV_g\leq\int_M u^{p-1}\Lap u\,dV_g+\alpha^2 p\norm{u}_p^p.
\end{equation*}
\end{lemma}

\begin{proof}
Compute, using integration by parts,
\begin{equation}\label{eq.davies}
\begin{split}
\int_M u^{p-1}L u\,dV_g &=-\int_M \nabla(\phi u)\nabla(\phi^{-1}u^{p-1})dV_g\\ 
&=\int_M \Big(\alpha^2 u^p\abs{\nabla\psi}^2-\alpha(p-2)u^{p-1}\nabla u\cdot\nabla\psi -(p-1)u^{p-2}\abs{\nabla u}^2\Big)dV_g\\
&\leq \alpha^2\norm{u}_p^p+\abs{\alpha}(p-2)\int_M u^{p-1}\abs{\nabla u}\,dV_g -\int_M \nabla u^{p-1}\nabla u\,dV_g\\
&\leq \big(\alpha^2+\tfrac{\abs{\alpha}(p-2)}{2s}\big)\norm{u}_p^p + \big(1-\tfrac{\abs{\alpha}(p-2)s}{2(p-1)}\big)\int_M u^{p-1}\Lap u\,dV_g,
\end{split}
\end{equation}
where the last line follows by estimating
\begin{align*}
2\int_M u^{p-1}\abs{\nabla u}\,dV_g &\leq s\int_M \big(u^{p/2-1}\abs{\nabla u}\big)^2dV_g + s^{-1}\int_M u^p dV_g\\
&= -\tfrac{s}{p-1}\int_M u^{p-1}\Lap u\,dV_g + s^{-1}\norm{u}_p^p.
\end{align*}
For $p>2$, the claimed inequality follows from \eqref{eq.davies} by defining $s:=\frac{p-1}{\abs{\alpha}(p-2)}$ and estimating the coefficient in front of $\norm{u}_p^p$ as follows,
\begin{equation*}
\big(\alpha^2+\tfrac{\abs{\alpha}(p-2)}{2s}\big) = \tfrac{\alpha^2}{2}\big(2+\tfrac{(p-2)^2}{p-1}\big) = \tfrac{\alpha^2}{2}\big(p+\tfrac{2-p}{p-1}\big)\leq \tfrac{\alpha^2}{2}p.
\end{equation*}
For $p=2$, we obtain from \eqref{eq.davies},
\begin{align*}
\int_M u^{p-1}L u\,dV_g &\leq \alpha^2\norm{u}_p^p + \int_M u^{p-1}\Lap u\,dV_g\\
&\leq \tfrac{\alpha^2}{2}p\norm{u}_p^p + \tfrac{1}{2}\int_M u^{p-1}\Lap u\,dV_g,
\end{align*}
where we added $-\tfrac{1}{2}\int_M u^{p-1}\Lap u\,dV_g = \tfrac{p-1}{2}\int_M u^{p-2}\abs{\nabla u}^2\,dV_g\geq 0$ in the last step.
\end{proof}

In the following, let $(M,g(t))$ be a solution to the Ricci flow on $[0,T)$ and let $\psi:M\times[0,T)\to\RR$ be a smooth function satisfying $\abs{\nabla\psi_t}_{g(t)}\leq 1$, where $\psi_t(\cdot)=\psi(\cdot,t)$. For such a $\psi$, define $L_t:=\phi_t^{-1}\Lap_{g(t)}\phi_t$ with $\phi_t=e^{\alpha\psi_t}$ for some $\alpha\in\RR$.

\begin{cor}[$L^p$-logarithmic Sobolev inequality involving the weighted operator $L$]\label{cor.logsob1}
Let $(M^n,g(t))$ be a solution to the Ricci flow in dimension $n\geq 3$ on $[0,T)$, $T<\infty$, and let $L_t=\phi_t^{-1}\Lap_{g(t)}\phi_t$ be as above. For every $\eps>0$, $t\in[0,T)$, $2\leq p<\infty$, and $0\leq u\in C^{\infty}_c(M)$, there holds
\begin{equation}\label{eq.logsob}
\begin{split}
\int_M u^p\log u\, dV_{g(t)} &\leq -\eps\int_M u^{p-1}L_t u\,dV_{g(t)} +\tfrac{p-1}{2p^2}\eps\int_M R_{g(t)}u^p\,dV_{g(t)}\\
&\quad\, +\widehat{\gamma}(\eps,p,t)\norm{u}_p^p +\norm{u}_p^p\log\norm{u}_p,
\end{split}
\end{equation}
where $\widehat{\gamma}(\eps,p,t):=\tfrac{1}{p}\big({-\frac{n}{2}}\log \eps + A+B\big(t+\tfrac{\eps}{2}\big)\big)+\tfrac{\eps\alpha^2p}{2}$ with $A$, $B$ as in Proposition \ref{prop.logsob}.
\end{cor}

\begin{proof}
This follows directly upon plugging Lemma \ref{lemma.davies1} in the form
\begin{equation*}
-\tfrac{1}{2}\eps\int_M u^{p-1}\Lap u\,dV_{g(t)}\leq -\eps\int_M u^{p-1}L_t u\,dV_{g(t)}+\tfrac{\eps\alpha^2p}{2}\norm{u}_p^p
\end{equation*}
into Lemma \ref{lemma.logsob} and estimating
\begin{align*}
\widetilde{\gamma}(\eps,p,t)&=\tfrac{1}{p}\big({-\tfrac{n}{2}}\log\eps-\tfrac{n}{2}\log\big(\tfrac{2(p-1)}{p}\big) + A+B\big(t+\tfrac{(p-1)}{2p}\eps\big)\big)\\
&\leq \tfrac{1}{p}\big({-\tfrac{n}{2}}\log\eps + A+B\big(t+\tfrac{\eps}{2}\big)\big).\qedhere
\end{align*}
\end{proof}

This corollary will be used to prove ultracontractivity estimates in the spirit of Davies \cite{D87} for the weighted heat operator $\dt-L_t$, that is, we show that its semigroup is a contraction semigroup from $L^2$ to $L^\infty$ (see Subsection \ref{subsec.ultra}). The main difference to the static case result of Davies is the presence of the scalar curvature term in \eqref{eq.logsob} requiring some subtle modifications of his arguments. Moreover, the duality argument used by Davies to show that this semigroup is also a contraction semigroup from $L^1$ to $L^2$ does not work in our setting of an evolving manifold. Hence, we need to develop new estimates for this step (see Subsection \ref{subsec.L12}), which require an $L^p$-logarithmic Sobolev inequality for $1<p<2$, derived in Corollary \ref{cor.logsob2} below. We first prove a result similar to Lemma \ref{lemma.davies1}.

\begin{lemma}\label{lemma.davies2}
For every complete manifold $(M^n,g)$, $0\leq u\in C^{\infty}_c(M)$, $1< p\leq 2$, and $L=\phi^{-1}\Lap_g\phi$ as before, we have
\begin{equation*}
2\int_M u^{p-1}Lu\,dV_g\leq\int_M u^{p-1}\Lap u\,dV_g+\alpha^2 \tfrac{p}{p-1} \norm{u}_p^p.
\end{equation*}
\end{lemma}

\begin{proof}
Set $v:= u^{p-1}$ and $q=\frac{p}{p-1}>2$. For $L=\phi^{-1}\Lap_g\phi$, we set $L^*:=\phi\Lap_g\phi^{-1}$. Now, applying Lemma \ref{lemma.davies1} to $v$ and $L^*$, we obtain
\begin{align*}
2\int_M u^{p-1}Lu \,dV_g &=2\int_M u L^*(u^{p-1})\, dV_g =2\int_M v^{q-1}L^*v\, dV_g\\
&\leq \int_M v^{q-1}\Lap v \, dV_g + \alpha^2 q\norm{v}_q^q = \int_M v\Lap (v^{q-1})\, dV_g + \alpha^2 q\norm{v}_q^q\\
&=\int_M u^{p-1}\Lap u\,dV_g +\alpha^2 \tfrac{p}{p-1} \norm{u}_p^p.\qedhere
\end{align*}
\end{proof}

\begin{cor}[$L^p$-logarithmic Sobolev inequality for $L$ with $1<p\leq 2$]\label{cor.logsob2}
Let $(M^n,g(t))$ be a solution to the Ricci flow on $[0,T)$ with $n\geq 3$ and let $L_t=\phi_t^{-1}\Lap_{g(t)}\phi_t$ be as above. For every $\eps>0$, $t\in[0,T)$, $1< p\leq 2$, and $0\leq u\in C^{\infty}_c(M)$, the Sobolev inequality \eqref{eq.logsob} holds with
\begin{equation*}
\widehat{\gamma}(\eps,p,t):=\tfrac{1}{p}\big({-\tfrac{n}{2}}\log\big(\tfrac{2(p-1)}{p}\,\eps\big)+ A+B\big(t+\tfrac{\eps}{4}\big)\big)+\tfrac{\eps\alpha^2p}{2(p-1)},
\end{equation*}
where $A$, $B$ are as in Proposition \ref{prop.logsob} above.
\end{cor}

\begin{proof}
Identical to the proof of Corollary \ref{cor.logsob1}, but using Lemma \ref{lemma.davies2} instead of Lemma \ref{lemma.davies1}.
\end{proof}

\subsection{Ultracontractivity Estimates}\label{subsec.ultra}
Here we prove that $\dt-L_t$ (with $L_t=\phi_t^{-1}\Lap_{g(t)}\phi_t$ as in the last subsection, i.e. $\phi_t=e^{\alpha\psi_t}$ for $\alpha\in\RR$ and smooth $\psi_t:M\to\RR$ satisfying $\abs{\nabla\psi_t}_{g(t)}\leq 1$) has an ultracontractive semigroup along a compact Ricci flow. This is stated more precisely in the following Lemma.

\begin{lemma}[Ultracontractivity estimates for the weighted heat operator $\dt-L_t$]\label{lemma.ultra}
Let $(M^n,g(t))$ be a solution to the Ricci flow $\dt g=-2\Rc$ on a \emph{positive and finite} time interval $[0,T)$ and assume that the underlying manifold is \emph{closed} (i.e.~compact and without boundary) and has dimension $n\geq 3$. Let $0\leq u\in C^\infty(M\times[t_0,t_1])$ be a solution of the weighted heat equation $\dt u=L_tu$, where $[t_0,t_1]\in[0,T)$. Then the $L^\infty$-norm of $u(t_1)$ (taken with respect to $g(t_1)$) is controlled by the $L^2$-norm of $u(t_0)$ (taken with respect to $g(t_0)$) via the following estimate
\begin{equation}\label{eq.ultra}
\norm{u(\cdot,t_1)}_{\infty,g(t_1)}\leq \frac{C_1}{(t_1-t_0)^{n/4}}\,e^{2\alpha^2(t_1-t_0)}\norm{u(\cdot,t_0)}_{2,g(t_0)},
\end{equation}
where $C_1$ depends only on $n$, $T$, and $(M,g_0)$.
\end{lemma}

\begin{proof}
We modify the ideas of Davies \cite{D87} in such a way that they work under Ricci flow. Set 
\begin{equation*}
\eps(q):=8(t_1-t_0)q^{-2}
\end{equation*}
and define $p(t)\geq 2$ for $t\in[t_0,t_1)$ by the implicit formula
\begin{equation*}
t=t_0+\int_2^p \tfrac{\eps(q)}{q}\,dq=t_1-4(t_1-t_0)p^{-2}.
\end{equation*}
In particular, we have $p(t_0)=2$ and $p(t)\to\infty$ as $t\to t_1$. With $\dt dV_{g(t)}=-R_{g(t)}dV_{g(t)}$ and $p':=\dt p=\frac{p}{\eps(p)}$, we compute, using the notation $\norm{u}_{p}=\norm{u(\cdot,t)}_{p(t),g(t)}$,
\begin{align*}
\frac{d}{dt}\norm{u}_{p} &= \frac{d}{dt}\bigg(\Big(\int_M u^{p(t)}(\cdot,t)dV_{g(t)}\Big)^{1/p(t)}\bigg)\\
&=-\tfrac{p'}{p^2}\norm{u}_{p}\log\norm{u}_p^p+\tfrac{1}{p}\norm{u}_p^{1-p}\Big( p'\int_M u^p\log u\,dV_{g(t)} +p\int_M (u^{p-1}\dt u-R_{g(t)}u^p)\,dV_{g(t)}\Big)\\
&=-\tfrac{1}{\eps(p)}\norm{u}_{p}\log\norm{u}_p+\norm{u}_p^{1-p}\Big( \tfrac{1}{\eps(p)} \int_M u^p\log u\,dV_{g(t)} +\int_M u^{p-1}(L_t-R_{g(t)})u\,dV_{g(t)}\Big).
\end{align*}
Hence, by plugging in \eqref{eq.logsob}, we find
\begin{equation}\label{eq.dt1}
\begin{split}
\frac{d}{dt}\norm{u}_{p} &\leq (\tfrac{p-1}{2p^2}-1)\norm{u}_p^{1-p}\int_M R_{g(t)}u^p\,dV_{g(t)} +\tfrac{\widehat{\gamma}(\eps(p),p,t)}{\eps(p)}\norm{u}_p\\
&\leq \big(\max_{M}R_{g(t)}^{-}+\tfrac{\widehat{\gamma}(\eps(p),p,t)}{\eps(p)}\big)\norm{u}_p\\
&\leq \big(\max_{M}R_{g(0)}^{-}+\tfrac{\widehat{\gamma}(\eps(p),p,T)}{\eps(p)}\big)\norm{u}_p,
\end{split}
\end{equation}
where $R_{g(t)}^{-}:=\max\{-R_{g(t)},0\}$ and $\widehat{\gamma}(\eps(p),p,t)$ is defined as in Corollary \ref{cor.logsob1}. The second line follows using $-1\leq(\tfrac{p-1}{2p^2}-1)\leq-\frac{7}{8}$, and the last line is a consequence of the well-known fact that the minimum of the scalar curvature is non-decreasing along a compact Ricci flow (and thus $R_{g(t)}^{-}$ is non-increasing) combined with the obvious monotonicity of $\widehat{\gamma}(\eps,p,t)$ in $t$.\\

Next, we define 
\begin{equation*}
N(t):=\int_2^{p(t)}\tfrac{\widehat{\gamma}(\eps(q),q,T)}{q}\,dq + (t-t_0)\max_{M}R_{g(0)}^{-},
\end{equation*}
which satisfies $N(t_0)=0$ and has the derivative
\begin{equation*}
\frac{dN}{dt} = \frac{\widehat{\gamma}(\eps(p),p,T)}{p}\cdot p' +\max_{M}R_{g(0)}^{-}= \frac{\widehat{\gamma}(\eps(p),p,T)}{\eps(p)} +\max_{M}R_{g(0)}^{-}.
\end{equation*}
Therefore, by \eqref{eq.dt1},
\begin{equation*}
\frac{d}{dt}\Big(\norm{u}_p\, e^{-N(t)}\Big) = e^{-N(t)}\Big(\frac{d}{dt}\norm{u}_p - \frac{dN}{dt}\cdot\norm{u}_p\Big)\leq 0,
\end{equation*}
or equivalently 
\begin{equation*}
\norm{u(\cdot,t)}_{p(t),g(t)}\leq e^{N(t)}\norm{u(\cdot,t_0)}_{2,g(t_0)}
\end{equation*}
for all $t\in[t_0,t_1)$. Taking a limit as $t\to t_1$, we obtain
\begin{equation}
\norm{u(\cdot,t_1)}_{\infty,g(t_1)}\leq e^{N(t_1)}\norm{u(\cdot,t_0)}_{2,g(t_0)}.
\end{equation}
The claim now follows from
\begin{align*}
N(t_1) &= \int_2^\infty \tfrac{\widehat{\gamma}(\eps(p),p,T)}{p}\,dp + (t_1-t_0)\max_{M}R_{g(0)}^{-}\\
&=\int_2^\infty \Big(\tfrac{1}{p^2}\big({-\tfrac{n}{2}}\log\eps(p) + A+B\big(T+\tfrac{\eps(p)}{2}\big)\big)+\tfrac{\eps(p)\alpha^2}{2}\Big)dp + (t_1-t_0)\max_{M}R_{g(0)}^{-}\\
&=\int_2^\infty \tfrac{1}{p^2}\big({-\tfrac{n}{2}}\log(8(t_1-t_0))+A+BT+4\alpha^2(t_1-t_0)\big)dp\\
&\quad\,+\int_2^\infty \tfrac{1}{p^4}\big(4B(t_1-t_0)\big)dp + \int_2^\infty \tfrac{n\log p}{p^2}\,dp + (t_1-t_0)\max_{M}R_{g(0)}^{-},
\end{align*}
which, using 
\begin{equation*}
\int_2^\infty \tfrac{n\log p}{p^2}\,dp = -\tfrac{n}{p}(1+\log p)\Big|_2^\infty=\tfrac{n}{2}(1+\log 2),
\end{equation*}
integrates to
\begin{align*}
N(t_1) &= \tfrac{1}{2}\big({-\tfrac{n}{2}}\log(8(t_1-t_0))+A+BT+4\alpha^2(t_1-t_0)\big)\\
&\quad\,+ \tfrac{1}{24}\big(4B(t_1-t_0)\big)+\tfrac{n}{2}(1+\log 2)+ (t_1-t_0)\max_{M}R_{g(0)}^{-}\\
&\leq-\tfrac{n}{4}\log(t_1-t_0)+2\alpha^2(t_1-t_0)+C.
\end{align*}
Note that $C:=\big(\frac{2}{3}B+\max_{M}R_{g(0)}^{-}\big)T+\tfrac{1}{2}A+\tfrac{n}{2}$ depends only on $n$, $T$, and $(M,g_0)$ and setting $C_1:=e^C$ then yields \eqref{eq.ultra}.
\end{proof}

\begin{rem}
If the scalar curvature $R_{g(0)}$ is non-negative and positive at some point, then $B$ in Proposition \ref{prop.logsob} can be chosen to be zero (see \cite{Y07,Z10}). But then also $\max_{M}R_{g(0)}^{-}=0$ and therefore the constant $C_1$ in Lemma \ref{lemma.ultra} is independent of $T$. 
\end{rem}

\subsection{Estimating the $L^2$-Norm by the $L^1$-Norm}\label{subsec.L12}
Here we show that the semigroup of $\dt-L_t$ is also a contraction semigroup from $L^1$ (at some time) to $L^2$ (at a later time). On an evolving manifold, we cannot use a duality argument as in Davies \cite{D87} but instead repeat the strategy from above with suitable modifications. In particular, as we are in the case $1<p<2$, we have to use Corollary \ref{cor.logsob2} instead of Corollary \ref{cor.logsob1}, which will force us to chose $\eps(q)$ differently, but otherwise the argument is actually quite similar. We have the following estimate.

\begin{lemma}[Contraction estimates from $L^1$ to $L^2$]\label{lemma.L12}
Let $(M^n,g(t))$ be a solution to the Ricci flow on a \emph{positive and finite} time interval $[0,T)$ and assume that $M$ is \emph{closed} and has dimension $n\geq 3$. Let $0\leq u\in C^\infty(M\times[t_0,t_1])$ be a solution of the weighted heat equation $\dt u=L_tu$, where $[t_0,t_1]\in[0,T)$. Then the $L^2$-norm of $u(t_1)$ (taken with respect to $g(t_1)$) is controlled by the $L^1$-norm of $u(t_0)$ (taken with respect to $g(t_0)$) as follows
\begin{equation}\label{eq.L12}
\norm{u(\cdot,t_1)}_{2,g(t_1)}\leq \frac{C_2}{(t_1-t_0)^{n/4}}\,e^{2\alpha^2(t_1-t_0)}\norm{u(\cdot,t_0)}_{1,g(t_0)},
\end{equation}
where $C_2$ depends only on $n$, $T$, and $(M,g_0)$.
\end{lemma}

\begin{proof}
We follow the proof of Lemma \ref{lemma.ultra} but this time we set
\begin{equation*}
\eps(q):=\frac{(t_1-t_0)}{\log 2-\frac{1}{2}}\cdot\frac{q-1}{q}.
\end{equation*}
We define $p(t)\in[1,2]$ for $t\in[t_0,t_1]$ by the implicit formula
\begin{equation*}
t=t_0+\int_1^p \frac{\eps(q)}{q}\,dq =t_0+\frac{\log p+\frac{1}{p}-1}{\log 2-\frac{1}{2}}\,(t_1-t_0),
\end{equation*}
which implies $p(t_0)=1$ and $p(t_1)=2$. Now, we follow the computation of \eqref{eq.dt1} in the proof of Lemma \ref{lemma.ultra} above, using Corollary \ref{cor.logsob2} instead of Corollary \ref{cor.logsob1}. This gives
\begin{equation}\label{eq.dt2}
\frac{d}{dt}\norm{u}_{p} \leq \big(\max_{M}R_{g(0)}^{-}+\tfrac{\widehat{\gamma}(\eps(p),p,T)}{\eps(p)}\big)\norm{u}_p,
\end{equation}
where $\widehat{\gamma}(\eps(p),p,t)$ is now given by Corollary \ref{cor.logsob2}. Setting 
\begin{equation*}
N(t):=\int_1^{p(t)}\tfrac{\widehat{\gamma}(\eps(q),q,T)}{q}\,dq + (t-t_0)\max_{M}R_{g(0)}^{-},
\end{equation*}
implies again $\frac{d}{dt}(\norm{u}_p\, e^{-N(t)})\leq 0$, from which we conclude in particular the estimate
\begin{equation}
\norm{u(\cdot,t_1)}_{2,g(t_1)}\leq e^{N(t_1)}\norm{u(\cdot,t_0)}_{1,g(t_0)}.
\end{equation}
To finish the proof, we have to compute
\begin{equation*}
N(t_1) = \int_1^2 \tfrac{\widehat{\gamma}(\eps(p),p,T)}{p}\,dp + (t_1-t_0)\max_{M}R_{g(0)}^{-}.
\end{equation*}
Writing $\eps(p)=c(t_1-t_0)\frac{p-1}{p}$ with $c=(\log 2-\frac{1}{2})^{-1}$, the integral becomes
\begin{align*}
\int_1^2 \tfrac{\widehat{\gamma}(\eps(p),p,T)}{p}\,dp &=\int_1^2 \Big(\tfrac{1}{p^2}\big({-\tfrac{n}{2}}\log\big(\tfrac{2(p-1)}{p}\,\eps(p)\big)+ A+B\big(T+\tfrac{\eps(p)}{4}\big)\big)+\tfrac{\eps(p)\alpha^2}{2(p-1)}\Big)dp\\
&=\int_1^2 \tfrac{1}{p^2}\big({-\tfrac{n}{2}}\log(t_1-t_0)-\tfrac{n}{2}\log(2c)+A+BT\big)\,dp\\
&\quad\,+\int_1^2 \tfrac{p-1}{4p^3}\big(Bc(t_1-t_0)\big)\,dp+\int_1^2 \tfrac{1}{2p}\big(\alpha^2c(t_1-t_0)\big)\,dp\\
&\quad\,+\int_1^2 \tfrac{1}{p^2}\big({-n}\log\big(\tfrac{(p-1)}{p}\big)\big)\,dp,
\end{align*}
which integrates to
\begin{align*}
N(t_1) &=\tfrac{1}{2}\big({-\tfrac{n}{2}}\log(t_1-t_0)-\tfrac{n}{2}\log(2c)+A+BT\big)+\tfrac{1}{32}Bc(t_1-t_0)\\
&\quad\,+\tfrac{\log 2}{2}\alpha^2c(t_1-t_0) + \tfrac{n}{2}(1+\log 2) + (t_1-t_0)\max_{M}R_{g(0)}^{-}.
\end{align*}
The only nontrivial integration is the following,
\begin{align*}
\int_1^2 \tfrac{1}{p^2}\big({-n}\log\big(\tfrac{(p-1)}{p}\big)\big)\,dp &=
-\tfrac{n}{p}\Big((p-1)\big(\log\big(\tfrac{(p-1)}{p}\big)-1\big)\Big)\Big|_1^2\\
&= \tfrac{n}{2}(1+\log 2)+n\lim_{p\to1}\Big((p-1)\log\big(\tfrac{(p-1)}{p}\big)\Big)\\
&=\tfrac{n}{2}(1+\log 2),
\end{align*}
where the limit vanishes according to L'H\^{o}pital's rule. Now, setting
\begin{equation*}
C:=\Big(\big(\tfrac{1}{2}+\tfrac{1}{32\log 2-16}\big)B+\max_{M}R_{g(0)}^{-}\Big)T + \tfrac{A}{2}+\tfrac{n}{4}\log(2\log 2-1)+\tfrac{n}{2}(1+\log 2),
\end{equation*}
which depends only on $n$, $T$ and $(M,g(0))$, we find
\begin{equation*}
N(t_1)\leq-\tfrac{n}{4}\log(t_1-t_0)+\tfrac{\log 2}{2\log 2-1}\alpha^2(t_1-t_0) + C,
\end{equation*}
and the claimed estimate \eqref{eq.L12} follows, setting $C_2:=e^{C}$ and noting that $\frac{\log 2}{2\log 2-1}<2$.
\end{proof}

\begin{rem}
As in the last subsection, if the scalar curvature $R_{g(0)}$ is non-negative and positive at some point, the constant $C_2$ is independent of $T$, since $B$ and $\max_{M}R_{g(0)}^{-}$ vanish. 
\end{rem}

\subsection{Proofs of Theorem \ref{main.theorem} and Corollary \ref{cor.main}}

Combining the Lemmas \ref{lemma.ultra} and \ref{lemma.L12}, we obtain a proof of the Main Theorem \ref{main.theorem}.

\begin{proof}[Proof of Theorem \ref{main.theorem}]
Let $0\leq u\in C^{\infty}(M\times[s,t])$ be a solution of the weighted heat equation $\dt u= L_t u$ with $L_t$ as above. Then, using one after another Lemma \ref{lemma.ultra} (with $[t_0,t_1]=[\frac{s+t}{2},t]$) and Lemma \ref{lemma.L12} (with $[t_0,t_1]=[s,\frac{s+t}{2}]$), we find
\begin{equation*}
\norm{u(\cdot,t)}_{\infty,g(t)} \leq \frac{2^{n/4}C_1}{(t-s)^{n/4}}\,e^{\alpha^2(t-s)}\norm{u(\cdot,\tfrac{s+t}{2})}_{2,g(\frac{s+t}{2})} \leq \frac{2^{n/2}C_1 C_2}{(t-s)^{n/2}}\,e^{2\alpha^2(t-s)}\norm{u(\cdot,s)}_{1,g(s)}.
\end{equation*}
Since
\begin{equation*}
u(x,t)=\int_M K(x,t;y,s)u(y,s) dV_{g(s)}(y)
\end{equation*}
for the fundamental solution of $\dt-L_t$, this is equivalent to the claimed estimate \eqref{eq.mainlemma} with $C=2^{n/2}C_1 C_2$ (which depends only on $n$, $T$ and the initial metic $g(0)$).
\end{proof}

The Gaussian bounds in Corollary \ref{cor.main} can now be obtained from the estimate is Theorem \ref{main.theorem} by choosing the right $\alpha$ and $\psi$.

\begin{proof}[Proof of Corollary \ref{cor.main}]
i) To make the notation more transparent, assume that we want to prove the estimate \eqref{main.eq1} for $0\leq s_0 < t_0 < T$ and $x_0,y_0\in M$. In this first step, we let $\psi_t\equiv\psi$ be time-independent. We first set
\begin{equation}
\psi(z):=\tfrac{1}{\mu}\min\{d_{g(t_0)}(z,y_0),d_{g(t_0)}(x_0,y_0)\}
\end{equation}
with $\mu$ defined by
\begin{equation*}
\mu:=\sup_{\lambda\in[s_0,t_0]}\,\sup_{M\setminus L}\,\abs{\nabla d_{g(t_0)}(y_0,\cdot)}_{g(\lambda)},
\end{equation*}
and $L$ being the set where $d_{g(t_0)}(y_0,\cdot)$ is not differentiable. We also set
\begin{equation}
\alpha:=\frac{1}{4(t_0-s_0)}(\psi(y_0)-\psi(x_0)).
\end{equation}
If $\psi$ would be a permitted weight function, then we would obtain the following. The fundamental solutions $H(x,t;y,s)$ of $\dt-\Lap_{g(t)}$ and $K(x,t;y,s)$ of $\dt-L_t$ are related by
\begin{equation*}
H(x,t;y,s)=\phi(x)K(x,t;y,s)\phi(y)^{-1} = K(x,t;y,s)\,e^{\alpha(\psi(x)-\psi(y))},
\end{equation*}
and thus applying \eqref{eq.mainlemma} yields
\begin{equation*}
H(x,t;y,s)\leq \frac{C}{(t-s)^{n/2}}\,e^{2\alpha^2(t-s)+\alpha(\psi(x)-\psi(y))}.
\end{equation*}
In particular, we obtain
\begin{align*}
H(x_0,t_0;y_0,s_0) &\leq \frac{C}{(t_0-s_0)^{n/2}}\,e^{2\alpha^2(t_0-s_0)+\alpha(\psi(x_0)-\psi(y_0))}\\
&=\frac{C}{(t_0-s_0)^{n/2}}\,e^{-\frac{(\psi(x_0)-\psi(y_0))^2}{8(t_0-s_0)}}\\
&= \frac{C}{(t_0-s_0)^{n/2}}\,e^{-\frac{d_{g(t_0)}(x_0,y_0)^2}{8\mu(t_0-s_0)}}.
\end{align*}

We therefore would indeed have \eqref{main.eq1} if $\psi$ would be a permitted weight function -- but it is not smooth. However, $\psi$ satisfies $\abs{\nabla\psi}_{g(\lambda)}\leq 1$ for all $\lambda\in[s_0,t_0]$ in the weak sense that $\abs{\psi(z_1)-\psi(z_2)}\leq d_{g(\lambda)}(z_1,z_2)$ and it is constant outside a fixed geodesic ball around $y_0$ with radius $d_{g(t_0)}(x_0,y_0)$. We can thus approximate it by $C^{\infty}$ functions $\psi_k$ satisfying $\abs{\nabla\psi_k}_{g(\lambda)}\leq 1$ for all $\lambda\in[s_0,t_0]$ and uniformly converging to $\psi$. This finishes the proof.\\

ii) Now, we let $\psi_t$ be time-dependent. Assume aganin that we want to prove the estimate \eqref{main.eq2} for fixed $0\leq s_0 < t_0 < T$ and $x_0,y_0\in M$, we set
\begin{equation}
\psi_t(z):=\min\{d_{g(t)}(z,y_0),d_{g(t)}(x_0,y_0)\}
\end{equation}
and 
\begin{equation}
\alpha:=\frac{1}{4(t_0-s_0)}(\psi_{t_0}(y_0)-\psi_{t_0}(x_0)).
\end{equation}
We note that $\psi_{t_0}(y_0)=0$, thus $\alpha\leq 0$. Setting $\widetilde{H}(x,t;y_0,s_0):=\phi_t(x)K(x,t;y_0,s_0)\phi_t(y_0)^{-1}$, we then obtain
\begin{equation}\label{eq.Htilde}
\dt \widetilde{H}=\alpha\dt(\psi_t(x)-\psi_t(y))\widetilde{H}+\Lap_{g(t)}\widetilde{H} \geq -4\eta\abs{\alpha}\widetilde{H}+\Lap_{g(t)}\widetilde{H},
\end{equation}
where $\eta$ is given by 
\begin{equation*}
\eta:=\tfrac{1}{4}\sup_{\lambda\in[s_0,t_0]} \sup_{z\in M}\; \max\big\{\tfrac{\partial}{\partial\sigma} \psi_\sigma(z) \big|_{\sigma=\lambda}, 0\big\} = \tfrac{1}{4}\sup_{\lambda\in[s_0,t_0]} \sup_{z\in M}\; \max\big\{\tfrac{\partial}{\partial\sigma} d_{g(\sigma)}(z,y_0) \big|_{\sigma=\lambda}, 0\big\}.
\end{equation*}
Since we know that for $d_{g(\lambda)}(z,y_0)\geq d_{g(\lambda)}(x_0,y_0)$ we have $\psi_\lambda(z)= d_{g(\lambda)}(x_0,y_0)$, we can replace the supremum over $z\in M$ with the supremum over $z$ satisfying $d_{g(\lambda)}(z,y_0)\leq d_{g(\lambda)}(x_0,y_0)$. It is important to pick $\eta$ as a constant, not depending on $t$ or $x$. We then obtain from \eqref{eq.Htilde} that
\begin{equation*}
\widetilde{H}_\eta:=e^{4\eta\abs{\alpha}(t-s_0)}\widetilde{H}
\end{equation*}
satisfies $\dt \widetilde{H}_\eta\geq\Lap_{g(t)}\widetilde{H}_\eta$ and because we know that $\widetilde{H}_\eta(\cdot,t;y_0,s_0)\to\delta_{y_0}$ as $t\searrow s_0$, we can show that the fundamental solution $H(x,t;y_0,s_0)$ is bounded above by $\widetilde{H}_\eta(x,t;y_0,s_0)$ by using a comparison principle argument. Hence, using \eqref{eq.mainlemma}, we obtain
\begin{align*}
H(x,t;y_0,s_0)\leq \widetilde{H}_\eta(x,t;y_0,s_0) &\leq K(x,t;y_0,s_0)\,e^{\alpha(\psi_t(x)-\psi_t(y_0))}\,e^{4\eta\abs{\alpha}(t-s_0)}\\
&\leq \frac{C}{(t-s_0)^{n/2}}\,e^{2\alpha^2(t-s_0)} e^{\alpha(\psi_t(x)-\psi_t(y_0))}\,e^{4\eta\abs{\alpha}(t-s_0)}.
\end{align*}
Estimating at $(x_0,t_0)$ and plugging in $\alpha$ as defined above, this yields
\begin{align*}
H(x_0,t_0;y_0,s_0) &\leq \frac{C}{(t_0-s_0)^{n/2}}\,e^{-\frac{(\psi_{t_0}(x_0)-\psi_{t_0}(y_0))^2}{8(t_0-s_0)}}\,e^{\eta\abs{\psi_{t_0}(y_0)-\psi_{t_0}(x_0)}}\\
&\leq \frac{C}{(t_0-s_0)^{n/2}}\,e^{-\frac{d_{g(t_0)}^2(x_0,y_0)}{8(t_0-s_0)}\,+\,\eta d_{g(t_0)}(x_0,y_0)}.
\end{align*}
Of course, $\psi_t(z)$ satisfies $\abs{\nabla\psi_\lambda}_{g(\lambda)}\leq 1$ for all $\lambda\in[s_0,t_0]$ in the weak sense as in the proof of part i) above and we can again use an approximation argument, approximating $\psi(z,t)$ by smooth functions $\psi_k(z,t)$ satisfying the derivative bounds in a strong sense and converging uniformly.
\end{proof}

As mentioned in the introduction, versions of Theorem \ref{thm.cty} and \ref{thm.cz} follow immediately from Corollary \ref{cor.main}, as uniform curvature bounds imply a bound on $\mu$ while nonnegative Ricci curvature implies $\eta=0$.

\section{Other intrinsic geometric flows}\label{sec.OtherFlows}

In this section, we study flows of the form $\dt g=-2\Sc$, where $\Sc=(S_{ij})$ is a symmetric two-tensor with trace $S = g^{ij}S_{ij}$. We also define the tensor quantity
\begin{equation}
\begin{aligned}
\mathcal{D}(\Sc,X) &:= \dt S-\Lap S -2\abs{S_{ij}}^2 +4(\nabla_i S_{ij})X_j -2(\nabla_j S)X_j \\
&\quad + 2R_{ij}X_iX_j - 2S_{ij}X_iX_j
\end{aligned}
\end{equation}
for a vector field $X\in\Gamma(TM)$ as introduced by the first author in \cite{M10}. We note that this generalises the Ricci flow studied in the previous section, since Ricci flow satisfies $\mathcal{D}(\Rc,X)\equiv 0$ for all vector fields $X$ on $M$. Other examples include List's extended Ricci flow \cite{L08}, harmonic Ricci flow \cite{M12}, twisted K\"ahler-Ricci flow \cite{CS16} on Fano manifolds, or Lorentzian mean curvature flow \cite{H99} on Lorentzian manifolds of nonnegative sectional curvatures.\\

While several of these flows had been studied before, the first systematic treatment of geometric flows satisfying $\mathcal{D}(\Sc,X) \geq 0$, $\forall X\in\Gamma(TM)$ appeared in \cite{M10}, where the first author obtained the monotonicity of a forward and backward reduced volume quantity for such flows. Later, the monotonicity of analogues of Perelman's $\mathcal{F}$-energy and his $\mathcal{W}$-entropy were proven for such flows, see e.g. \cite{H10, GPT13, FZ16}. In fact, it can be seen from a straight-forward (but slightly lengthy) computation that if $(M, g(t))$ is a solution to $\dt g=-2\Sc$ on a closed manifold of dimension $n \geq 3$ for $t \in [0,T)$, $\partial_t \tau = -1$, and $f$ satisfies
\begin{equation*}
\partial_t f = -\Lap f + \abs{\nabla f}^2 - S + \tfrac{n}{2\tau}
\end{equation*}
as well as the normalisation
\begin{equation}\label{eq.normalisationf}
\int_M \frac{e^{-f}}{(4\pi \tau)^{n/2}}\, dV_{g(t)} = 1
\end{equation}
then the $\mathcal{W}$-entropy
\begin{equation*}
\mathcal{W}\left(g,f,\tau\right) := \int_M \left[\tau \left( S_{g(t)} + \left|\nabla f\right|^2_{g(t)} \right) + f - n\right] \frac{e^{-f}}{\left(4\pi \tau\right)^{\frac{n}{2}}}dV_{g(t)}
\end{equation*}
satisfies
\begin{equation}\label{eq.Wmono}
\frac{d}{dt} \mathcal{W} = \int_M 2\tau \left( \left| S_{ij} + \mathrm{Hess}\left(f\right) - \frac{g}{2\tau} \right|^2_{g(t)} + \mathcal{D}\left(\Sc, -\nabla f\right) \right) \frac{e^{-f}}{\left(4\pi \tau\right)^{\frac{n}{2}}} dV_{g(t)}.
\end{equation}
In particular, if $\mathcal{D}\left(\Sc,-\nabla f\right) \geq 0$ then $\mathcal{W}$ is non-decreasing.\\

Just like in the Ricci flow case, this $\mathcal{W}$-monotonicity can be used to derive logarithmic Sobolev inequalities similar to the ones in Proposition \ref{prop.logsob}.

\begin{prop}[Log-Sobolev inequality along flows with $\mathcal{D}(\Sc,X)\geq 0$, cf. \cite{FZ16}]\label{prop.sobS}
Let $(M^n,g(t))$ be a compact solution to $\dt g=-2\Sc$ with $n\geq 3$ on $[0,T)$, $T<\infty$ satisfying \eqref{eq.D}. For all $\eps>0$ and each $t\in[0,T)$, there holds
\begin{equation*}
\int_M v^2\log v^2\, dV_{g(t)} \leq \eps\int_M\big(\abs{\nabla v}^2+\tfrac{1}{4} S_{g(t)}v^2\big)dV_{g(t)} +\gamma(\eps,t),
\end{equation*}
for all $0\leq v\in C^{\infty}_c(M)$ with $\norm{v}_2=1$. Here, $S_{g(t)}$ denotes the trace of $\Sc_{g(t)}$ and 
\begin{equation*}
\gamma(\eps,t):=-\tfrac{n}{2}\log \eps + A+B\big(t+\tfrac{\eps}{4}\big),
\end{equation*}
where $A$, $B$ are constants depending only on $(M,g(0))$ and $S_{g(0)}$.
\end{prop}

To make this article more self-contained, we give a proof of this proposition which follows \cite{Y07,FZ16} quite closely. We first note that for $f$ as in the discussion above, satisfying in particular \eqref{eq.normalisationf}, the function $u := (4\pi \tau)^{-n/4}e^{-f/2}$ satisfies the normalisation $\int_M u^2\, dV_{g(t)} = 1$. We then set
\begin{align*}
\mathcal{W}^\ast (g,u,\tau) &= \mathcal{W}(g,f,\tau) + \frac{n}{2}\ln(\tau) + \frac{n}{2}\ln(4\pi) + n\\
&= \int_M \Big[\tau (4\abs{\nabla u}^2_{g(t)} + S_{g(t)}u^2) - u^2 \ln(u^2)\Big]dV_{g(t)}
\end{align*}
as well as 
\begin{equation}\label{mod_gen_mu}
\mu^\ast(g,\tau) := \inf \limits_{u} \mathcal{W}^\ast(g,u,\tau),
\end{equation}	
where the infimum is taken over all $u$ satisfying the above normalisation. From the monotonicity \eqref{eq.Wmono} we obtain
for $\tau(t):=t^\ast + \sigma-t$ with $\sigma>0$ and $t\in[0,T)$
\begin{equation*}
\frac{d}{dt} \mathcal{W}^\ast (g,u,\tau) \geq \frac{n}{2} \frac{d}{dt} \ln(\tau),
\end{equation*}
and therefore
\begin{equation*}
\mu^\ast( g(t_1), \tau(t_1) ) \leq \mu^\ast( g(t_2), \tau(t_2) ) + \frac{n}{2}\ln\frac{\tau(t_1)}{\tau(t_2)}.
\end{equation*}
Setting $t_1=0$ and $t_2=t^\ast$, we find
\begin{equation}\label{eq.mumonotonicity}
\mu^\ast(g(0), t^\ast + \sigma) \leq \mu^\ast(g(t^\ast), \sigma) + \frac{n}{2}\ln \frac{t^\ast + \sigma}{\sigma}, \quad \forall t^\ast \in [0,T).
\end{equation}
Endowed with these preliminaries, we can now start the proof of the proposition.

\begin{proof}[Proof of Proposition \ref{prop.sobS}]
Let $C_S$ denote the $L^2$ Sobolev constant of $(M,g(0))$, i.e.
\begin{equation*}
C_S= C_S(M,g(0)) := \sup \limits_{u \in W^{1,2}(M)} \Big\{ \norm{u}_{2^\ast}- \frac{\norm{u}_2}{\Vol_{g(0)}(M)^{1/n}}\; ; \; \norm{\nabla u}_2 = 1 \Big\},
\end{equation*}
where $2^\ast = \frac{2n}{n-2}$ is the Sobolev conjugate of $2$. Using Jensen's inequality for concave functions with respect to the measure $v^2 dV$ and the assumption $\norm{v}_2 = 1$, we find
\begin{equation*}
\ln\left(\int_M v^{2^\ast} dV_{g(0)}\right) = \ln\left(\int_Mv^{2^\ast-2}  \, v^2 dV_{g(0)}\right) \geq \int_M  \ln\left(v^{2^\ast-2}\right) v^2dV_{g(0)}
\end{equation*}
and therefore, for $\beta>0$,
\begin{align*}
\int_M v^2 \ln\big(v^2\big)dV_{g(0)} &\leq \frac{2^\ast}{2^\ast-2} \ln\Big(\norm{v}_{2^\ast}^{2}\Big)\\
&\leq \frac{n}{2} \ln\Big(\big(C_S\norm{\nabla v}_2 + \Vol_{g(0)}(M)^{-1/n}\big)^2\Big)\\
&\leq \frac{n}{2}\ln(2) + \frac{n}{2}\ln\Big(C^2_S\norm{\nabla v}^2_2 + \Vol_{g(0)}(M)^{-2/n}\Big)\\
&\leq \frac{n}{2}\ln(2) + \frac{n}{2}\Big(\beta C^2_S\norm{\nabla v}^2_2 + \beta \Vol_{g(0)}(M)^{-2/n} - 1 - \ln(\beta)\Big)\\
& \leq \frac{n \beta C^2_S}{2}\int_M \Big(\abs{\nabla v}^2 + \frac{S_{g(0)}}{4}v^2\Big)dV_{g(0)}\\
&\quad- \frac{n}{2}(\ln(\beta) - \ln(2) + 1) + \frac{n \beta}{2} \Big(\Vol_{g(0)}(M)^{-2/n} - \frac{\min_M S_{g(0)}}{4}C^2_S\Big).
\end{align*}
In the fourth step we used that $\ln(x+y) \leq \beta x + \beta y - 1 - \ln(\beta)$ for all $x \geq 0$, $\beta > 0$, and $y > -x$ (see e.g. Lemma 3.2 in \cite{Y07}), and in the last step we used $S-\min_M S\geq 0$. We now pick
\begin{equation*}
\beta := \frac{8(t+\sigma)}{n C_S^2}
\end{equation*}
to obtain
\begin{align*}
\int_M v^2 \ln\big(v^2\big)dV_{g(0)} &\leq (t+\sigma)\int_M \Big(4\abs{\nabla v}_{g(0)}^2 + S_{g(0)}v^2\Big)dV_{g(0)} \\
&\quad - \frac{n}{2}\ln(t+\sigma) + \frac{n}{2}(2\ln(C_S)+\ln(n) - 2\ln(2) -1)\\
&\quad + (t+\sigma) \Big(4C_S^{-2}\Vol_{g(0)}(M)^{-2/n} - \min_M S_{g(0)}\Big)\\
&= (t+\sigma)\int_M \Big(4\abs{\nabla v}_{g(0)}^2 + S_{g(0)}v^2\Big)dV_{g(0)} \\
&\quad - \frac{n}{2}\ln(t+\sigma) + A + B(t+\sigma) - n\ln(2),
\end{align*}
where we set $A= \frac{n}{2}(2\ln(C_S)+\ln(n) -1)$ and $B= 4C_S^{-2}\Vol_{g(0)}(M)^{-2/n} - \min_M S_{g(0)}$. This formula is equivalent to
\begin{equation*}
\mu^\ast(g(0), t + \sigma) \geq \frac{n}{2}\ln(t+\sigma) - A - B(t+\sigma) +n\ln(2)
\end{equation*}
and so by the monotonicity formula \eqref{eq.mumonotonicity}
\begin{equation*}
\mu^\ast(g(t), \sigma) \geq \frac{n}{2}\ln(\sigma) - A - B(t+\sigma) +n\ln(2)
\end{equation*}
or equivalently
\begin{equation*}
\mu^\ast(g(t), \tfrac{\eps}{4}) \geq \frac{n}{2}\ln(\eps) - A - B(t+\tfrac{\eps}{4}) = -\gamma(\eps,t).
\end{equation*}
The last formula is obviously equivalent to the claim in the proposition.
\end{proof}

With Proposition \ref{prop.sobS} in hand, it is easy to prove Theorem \ref{thm.mainS}.

\begin{proof}[Proof of Theorem \ref{thm.mainS}]
We can follow the proof of Theorem \ref{main.theorem} given in Section \ref{sec.Ricci} verbatim, simply replacing the scalar curvature $R_{g(t)}$ with the tensor $S_{g(t)}$ and using Proposition \ref{prop.sobS} instead of Proposition \ref{prop.logsob}. In order to do so, we need 
\begin{equation}
\dt dV_{g(t)}=-S_{g(t)}dV_{g(t)},
\end{equation} 
which follows from the general variation formula for the volume element (see e.g. Proposition 1.5 in \cite{M06}) as well as the fact that the minimum of $S_{g(t)}$ is non-decreasing along a compact flow satisfying \eqref{eq.D} -- and hence $S_{g(t)}^{-}$ is non-increasing. This latter fact follows by taking $X=0$ in \eqref{eq.D}, which yields
\begin{equation*}
\dt S-\Lap S -2\abs{S_{ij}}^2 \geq 0,
\end{equation*}
and a simple maximum principle argument. Once the bounds from Theorem \ref{main.theorem} are proven, the bounds from Corollary \ref{cor.main} follow immediately as in Section \ref{sec.Ricci}.
\end{proof}

\makeatletter
\def\@listi{%
  \itemsep=0pt
  \parsep=1pt
  \topsep=1pt}
\makeatother
{\fontsize{10}{11}\selectfont

}

\printaddress

\end{document}